\documentclass[twoside,11pt]{article}
\usepackage{amsmath}
\usepackage{amsthm}
\usepackage{amssymb}
\usepackage{amscd}

\usepackage{latexsym}
\usepackage{amsfonts}

\usepackage[all]{xy}

\newcommand{\be}{\begin{equation}}
\newcommand{\ee}{\end{equation}}
\newcommand{\ra}{\rightarrow}

\newcommand{\bea}{\begin{eqnarray}}
\newcommand{\eea}{\end{eqnarray}}
\newcommand{\beas}{\begin{eqnarray*}}
\newcommand{\eeas}{\end{eqnarray*}}

\newcommand{\R}{{\mathbb R}}
\newcommand{\Z}{{\mathbb Z}}
\def\X{\mathfrak X}
\def\on{\operatorname}
\def\pa{\partial}
\def\Sec{\operatorname{Sec}}
\def\hol{\operatorname{hol}}
\def\xd{\on{d}}
\def\xdp{{\on{d}^\Pi}}

\def\cM{{\mathcal M}}
\def\A{{\mathcal A}}

\def\cR{{\mathcal R}}
\def\ra{\rightarrow}
\def\we{\wedge}
\def\cL{{\mathcal L}^\Pi}
\def\wt{\widetilde}
\def\ph{\phi}
\def\la{\langle}
\def\ran{\rangle}
\newcommand{\Ll}{{\pounds}}
\def\ot{\otimes}
\def\ti{\times}
\def\mod{\on{mod}}
\def\div{\on{div}}
\def\graph{\on{graph}}
\def\ch{\on{char}}
\def\ol{\overline}
\newcommand{\dd}{\mathrm{d}}

\mathchardef\za="710B  
\mathchardef\zb="710C  
\mathchardef\zg="710D  
\mathchardef\zd="710E  
\mathchardef\zve="710F 
\mathchardef\zz="7110  
\mathchardef\zh="7111  
\mathchardef\zvy="7112 
\mathchardef\zi="7113  
\mathchardef\zk="7114  
\mathchardef\zl="7115  
\mathchardef\zm="7116  
\mathchardef\zn="7117  
\mathchardef\zx="7118  
\mathchardef\zp="7119  
\mathchardef\zr="711A  
\mathchardef\zs="711B  
\mathchardef\zt="711C  
\mathchardef\zu="711D  
\mathchardef\zvf="711E 
\mathchardef\zq="711F  
\mathchardef\zc="7120  
\mathchardef\zw="7121  
\mathchardef\ze="7122  
\mathchardef\zy="7123  
\mathchardef\zvp="7124  
\mathchardef\zvr="7125 
\mathchardef\zvs="7126 
\mathchardef\zf="7127  
\mathchardef\zG="7000  
\mathchardef\zD="7001  
\mathchardef\zY="7002  
\mathchardef\zL="7003  
\mathchardef\zX="7004  
\mathchardef\zP="7005  
\mathchardef\zS="7006  
\mathchardef\zU="7007  
\mathchardef\zF="7008  
\mathchardef\zC="7009  
\mathchardef\zW="700A  

\def\y{\mathbf{y}}
\def\U{\mathbf{U}}
\def\uu{\mathbf{u}}

\def\wh{\widehat}
\def\ulg{\underline{\zg}}
\newcommand{\bk}[2]{\ensuremath{\langle #1 | #2 \rangle}}

\newcommand{\PP}{\operatorname{P}}

=1in \textwidth 15.6cm \topmargin 10mm\textheight23cm \advance\topmargin by -\headheight\advance\topmargin by
-\headsep

\newtheorem{theorem}{Theorem}[section]
\newtheorem{proposition}[theorem]{Proposition}

\theoremstyle{definition}
\newtheorem{example}[theorem]{Example}
\newtheorem{definition}[theorem]{Definition}
\newtheorem{remark}[theorem]{Remark}

\begin{document}
\title{Modular classes of skew algebroid relations \thanks{Research
supported by the Polish Ministry of Science and Higher Education under the grant N N201 416839.}}

\author{Janusz Grabowski \\ \\{\it Institute of Mathematics}\\
                {\it Polish Academy of Sciences}}
\maketitle
\begin{abstract} {\it Skew algebroid} is a natural generalization of the concept of Lie algebroid. In this paper, for a skew algebroid $E$, its {\it modular class} $\mod(E)$ is defined in the classical as well as in the supergeometric formulation. It is proved that there is a homogeneous nowhere-vanishing 1-density on $E^*$ which is invariant with respect to all Hamiltonian vector fields if and only if $E$ is modular, i.e. $\mod(E)=0$. Further, {\it relative modular class} of a subalgebroid is introduced and studied together with its application to holonomy, as well as {\it modular class of a skew algebroid relation}. These notions provide, in particular, a unified approach to the concepts of a modular class of a Lie algebroid morphism and that of a Poisson map.


\end{abstract}

\section{Introduction}
The concept of {\it modular class} of a Lie algebroid, proposed by Weinstein and introduced in \cite{ELW,We},
is a natural extension of that for a Poisson manifold \cite{Ko,We} and hence an analogue of the modular
automorphism group of a von Neumann algebra. It is well known that a Lie algebroid structure on a vector
bundle $E$ is canonically associated with a linear Poisson structure on the dual bundle $E^\ast$, and the modular class of $E$ can be viewed as the obstruction to the existence of a homogeneous measure on $E^\ast$ being invariant with respect to all Hamiltonian vector fields on the Poisson manifold $E^\ast$ \cite{We}. The question of the existence of a measure on $E^\ast$ invariant with respect to certain Hamiltonian vector fields of mechanical origins was studied recently by Marrero \cite{M}. In the framework of Lie-Rinehart algebras, the concept of modular class was developed by Huebschmann \cite{Hu}, and the role of Batalin-Vilkovisky algebras was emphasized in \cite{KS,Xu}.

It was then observed in \cite{GMM} that the modular class can be naturally defined also for a base-preserving morphism of Lie algebroids. A similar object has been studied in \cite{KSW} under the name {\it relative modular class}
and then extended to arbitrary Lie algebroid morphisms \cite{KLW}. In \cite{Fe} Fernandes constructed a
sequence of secondary characteristic classes of a Lie algebroid whose first element coincides with the modular
class (see also \cite{Ku}). Secondary characteristic classes of a base-preserving Lie algebroid morphism were
studied also by Vaisman \cite{Vai}.

Poisson manifolds were revisited recently by Caseiro and Fernandes in \cite{CF}, where modular class of a Poisson map was defined and studied in detail. The problem was that, although any Poisson structure $\zL$ on a manifold $M$ gives rise to a Lie algebroid structure on the cotangent bundle $T^\ast M$ and the modular class of this Lie algebroid is just twice the modular class of the Poisson manifold, a Poisson map $\zf:M_1\ra M_2$ does not in general induce any canonical Lie algebroid morphism between $T^\ast M_1$ and $T^\ast M_2$.

On the other hand, several concepts of objects generalizing Lie algebroids appeared recently in the
literature. One of them is the concept of  {\it general algebroid} or its skew-symmetric version, a {\it skew
algebroid}, introduced in \cite{GU3,GU} and applied in analytical mechanics in \cite{GG,GGU} as an extension of
Lie algebroid/groupoid description of Lagrangian and Hamiltonian formalisms proposed by Weinstein \cite{We1}. The interest to skew algebroids, for which the Jacobi identity valid for Lie algebroids is dropped, comes from the nonholonomic mechanics in which they provide a natural geometric framework (see \cite{GLMM} and references therein).

Our aim in this paper is twofold: first, to define modular class for a skew algebroid and, second, to extend
the concept of modular class of a Lie algebroid morphism to that of a linear relation between skew algebroids. The idea is very simple and comes from considering Poisson maps. Namely, if $\zf:M_1\ra M_2$ is a Poisson map between Poisson manifolds, then the tangent bundle \ $T\!\graph(\zf)$ of the graph of $\zf$ is a coisotropic subbundle in the direct product $TM_1\ti \ol{TM_2}\simeq T(M_1\ti \ol{M_2})$, where $\ol{M}$ denotes the manifold $M$ with the opposite Poisson structure. Hence, its annihilator $\cR=(T\graph(\zf))^\perp$ is a Lie subalgebroid in the direct product Lie algebroid $T^\ast M_1\ti \ol{T^\ast M_2}$, thus, by definition, a {\it Lie algebroid relation}. Hence, dealing with a Poisson map is just a particular case of dealing with a Lie algebroid relation for which modular class can be naturally defined, and the whole picture can be easily extended to the context of skew algebroids.

Note that this framework is not only unifying several concepts of modular class and generalizes the setting by dropping the Jacobi identity, but it fits also much better to the modern approach to various problems in geometry by considering relations rather than maps.

The paper is organized as follows. In Section 2 we present basic notions and facts from the theory of skew algebroids.
Modular class of a skew algebroid and characterizations of modularity are presented in Section 3 together with an application to nonholonomic mechanics. Section 4 contains an elementary proof that the modular class can be simply defined in the language of supergeometry as the (super)divergence of the de Rham vector field (homological vector field in the case of a Lie algebroid) with respect to a natural class of Berezinian non-oriented densities on the supermanifold $E[1]$. Relative modular class of a subalgebroid is defined in Section 5 and then used in Section 6 to a nice description of the corresponding holonomy along an admissible path. Finally, in Section 7, we present the modular class of a skew algebroid relation (and, more generally, a linear relation between skew algebroids) together with showing its unifying role in the theory.

\section{Skew algebroids}

Let $\zt:E\to M$ be a rank-$n$ vector bundle over an $m$-dimensional manifold $M$, and let $\zp:E^*\ra M$ be
its dual. Let $\mathcal A^i(E)=\on{Sec}(\wedge^iE)$, for $i=0,1,2,\dots$, denotes the module of sections of
the bundle $\wedge^iE$. Let us put $\mathcal A^i(E)=\{0\}$ for $i<0$, and denote with $\mathcal
A(E)=\bigoplus_{i\in \mathbb Z}\mathcal A^i(E)$ the Grassmann algebra of multisections of $E$. It is a graded
commutative associative algebra with respect to the wedge product. Similarly we define the Grassmann algebra
$\mathcal A(E^*)$ for the dual bundle $E^*$.

There are different equivalent ways to define a {\it skew algebroid} structure on $E$. Here we will list only
three of them. The notation is borrowed from \cite{GU, GGU, GG} and we refer to these papers for details. In
particular, we use affine coordinates $(x^a,\zx_i)$ on $E^*$ and the dual coordinates $(x^a,y^i)$ on $E$,
as associated with dual local bases, $(e_i)$ and $(e^i)$, of sections of $E$ and $E^*$, respectively.

\begin{definition}  A {\it skew algebroid} structure on $E$ is given by a linear bivector
field $\zP$ on $E^*$. In local coordinates,
\be\label{SA} \Pi =c^k_{ij}(x)\zx_k
\partial _{\zx_i}\otimes \partial _{\zx_j} + \zr^b_i(x) \partial _{\zx_i}
\wedge \partial _{x^b}\,,
\ee
where $c^k_{ij}(x)=-c^k_{ji}(x)$. The {\it opposite} algebroid $\ol{E}$ is the skew algebroid corresponding to
the bivector field $-\Pi$. If $\Pi$ is a Poisson tensor, we speak about a {\it Lie algebroid}.
\end{definition}

As the bivector field $\Pi$ defines a bilinear bracket $\{\cdot,\cdot\}_\Pi$ on the algebra $C^\infty(E^*)$ of
smooth functions on $E^*$ by $\{\zvf,\psi\}_{\zP}=\langle\zP,\xd\zvf\we\xd\psi\rangle$, where $\langle\cdot,\cdot\rangle$ stands for contractions, we get the following.

\begin{theorem}
A skew algebroid structure $(E,\Pi)$ can be equivalently defined as
\begin{itemize}
\item a skew-symmetric $\R$-bilinear bracket
$[\cdot ,\cdot]_\Pi $ on the space $\Sec(E)$ of sections of $E$, together with a vector bundle morphisms\ $\zr
\colon E\rightarrow T M$ ({\it the anchor}), such that
\be\label{qd} [X,fY]_\Pi =\zr(X)(f)Y +f [X,Y]_\Pi,
\ee
for all $f \in C^\infty (M)$, $X,Y\in \Sec(E)$,

\item or as a derivation $\xdp$ of degree 1 in the Grassmann algebra
$\A(E^*)$ (the {\it de Rham derivative}). The latter is a map $\xdp:\A(E^*)\ra\A(E^*)$ such that
$\xdp:\A^i(E^*)\ra\A^{i+1}(E^*)$ and, for $\za\in\A^a(E^*)$, $\zb\in\A^b(E^*)$, we have
\be\label{dR}\xdp(\za\we\zb)=\xdp\za\we\zb+(-1)^a\za\we\xdp\zb\,.\ee
\end{itemize}
\end{theorem}
The bracket $[\cdot,\cdot]_\Pi$ and the anchor $\zr$ are related to the bracket $\{\cdot,\cdot\}_{\zP}$
according to the formulae
\bea\label{drel}
        \zi([X,Y]_\Pi)&= \{\zi(X), \zi(Y)\}_{\zP},  \\
        \zp^*(\zr(X)(f))       &= \{\zi(X), \zp^*f\}_{\zP}\,,\label{drel1}
                                                   \eea
where we denoted with $\zi(X)$ the linear function on $E^*$ associated with the section $X$ of $E$, i.e.
$\zi(X)(e^\ast_p)=\la X(p),e^\ast_p\ran$ for each $e^\ast_p\in E^\ast_p$.

 \noindent The {\it exterior
derivative} $\xdp$ is determined by the formula
\be\label{dRd}(\xdp\zm)^v=[\zP,\zm^v],\ee
where $\zm^v$ is the natural vertical lift of a k-form $\zm\in\A^k(E^*)$ to a vertical k-vector field on $E^*$
and $[\cdot,\cdot]$ is the Schouten bracket of multivector fields.

\medskip
In local bases of sections and the corresponding local coordinates,
\bea\label{r1}
[e_i,e_j]_\Pi(x)&=&c^k_{ij}(x)e_k,\\
\zr(e_i)(x)&=&\zr^a_i(x)\pa_{x^a},\label{r2}\\
\xdp f(x)&=&\zr^a_i(x)\frac{\pa f}{\pa x^a}(x)e^i,\label{r3}\\
\xdp e^i(x)&=&c^i_{lk}(x)e^k\we e^l\,.\label{r4}
\eea
Of course, in general $(\xdp)^2\ne 0$, and $(\xdp)^2=0$ if and only if $\Pi$ is a Poisson tensor, thus our
skew algebroid is actually a Lie algebroid. If the skew algebroid satisfies the identity (e.g., it is a Lie algebroid)
\be\label{AL} \zr([X,X']_\zP)=[\zr(X),\zr(X')]\,,
\ee
we speak about an {\it almost Lie algebroid}. For an almost Lie algebroid, $(\xdp)^2=0$ as a map from
$\A^0(E^*)=C^\infty(M)$ to $\A^2(E^*)$, thus the standard first cohomology space $H^1(E)$ is well defined as well as
homotopies of admissible curves (see \cite{GJ}).

\begin{remark} (supergeometric picture) Note that the formula (\ref{dRd}) is equivalent to the well-known Cartan's formula for the de Rham
derivative. Moreover, as the Grassmann algebra $\A(E^*)$ can be understood as the algebra of smooth functions
on the graded (super)manifold $E[1]$ (an {\it N-manifold of degree} 1 in the terminology of \v Severa and Roytenberg \cite{Roy, Sev}),
following \cite{Va} we can view the de Rham derivative $\xdp$ as a vector field of degree 1 on $E[1]$. This
vector field is {\it homological}, $(\xdp)^2=0$, if and only if we are actually dealing with a Lie algebroid.
In local supercoordinates $(x,\y)$ associated canonically with our standard affine coordinates $(x,y)$ we have
\be\label{sdR} \xdp=\frac{1}{2}c^k_{ij}(x)\y^j\y^i\pa_{\y^k} + \zr^b_i(x)\y^i\partial _{x^b}\,.
\ee
\end{remark}

\bigskip For any section $X\in\Sec(E)$, the {\it Lie derivative} $\cL_X$, acting in $\A(E)$ and $\A(E^*)$, is defined in the standard way:  $\cL_X(f)=\zr(X)(f)$ for $f\in C^\infty(M)$,
$$\cL_X(Y_1\we\cdots\we Y_a)=\sum_iY_1\we\cdots\we[X,Y_i]_\Pi\we\cdots\we Y_a$$
for $Y_1\we\cdots\we Y_a\in\A^a(E)$, and
\be\label{LD}\la\cL_X(\za),Y_1\we\cdots\we Y_a\ran=\cL_X\la \za,Y_1\we\cdots\we Y_a\ran-\la\za,\cL_X(Y_1\we\cdots\we Y_a)\ran\,,\ee
for $\za\in\A^a(E^\ast)$, $a>0$. The formula (\ref{LD}) is equivalent to the Cartan's formula $\cL_X=i_X\xdp+\xdp i_X$.

\medskip
Recall that, given a skew algebroid $E$, we can associate with any $C^1$-function $H$ on $E^*$ its {\it
Hamiltonian vector field} $\X_H$ like in the standard case: $\X_H=i_{\xd\! H}\Pi$. In local coordinates,
\be\label{ham}
\X_H(x,\zx)=\left(c^k_{ij}(x)\zx_k\frac{\pa H} {\partial {\zx_i}}(x,\zx)- \zr^a_j(x)\frac{\pa H}{\partial
{x^a}}(x,\zx)\right) \partial _{\zx_j} + \zr^b_i(x) \frac{\pa H}{\partial {\zx_i}}(x,\zx)
\partial _{x^b}\,.
\ee
Another geometrical construction in the skew-algebroid setting is the \emph{complete lift of an algebroid
section} (cf. \cite{GU3,GU}). For every $C^1$-section $X=f^i(x)e_i\in\Sec(E)$ we can construct canonically a
vector field $\dd_T^\Pi(X)\in\Sec(T E)$ which in local coordinates reads as
\begin{equation}\label{eqn:tan_lift}
\dd_T^\Pi(X)(x,y) = f^i(x)\rho^a_i(x)\pa_{x^a} + \left( y^i \rho^a_i(x) \frac{\pa f^k}{\pa x^a}(x) +
c^k_{ij}(x)y^if^j(x) \right) \pa _{y^k}.
\end{equation}
The vector field $\dd_T^\Pi(X)$ is homogeneous (linear) with respect to $y$'s.

\medskip
An {\it $E$-connection} on a vector bundle $\zz:A\ra M$ is a `covariant derivative'
$\nabla:\Sec(E)\ti\Sec(A)\ra\Sec(A)$ such that $\nabla_{fX}Y=f\nabla_XY$ and
\be\label{cqd}\nabla_X(fY)=\zr(X)(f)Y+f\nabla_XY
\ee
for any $f\in C^\infty(M)$, $X\in\Sec(E)$, and $Y\in\Sec(A)$.
This means that $D=\nabla_X$ is a {\it quasi-derivation} (called also {\it derivative endomorphism, module derivation, covariant differential operator}, etc.; see historical remarks in \cite{KSMa} and references therein) on sections of $A$ with the anchor $\wh{D}=\zr(X)\in\X(M)$.
The same property has the operator $ad^\Pi_X=[X,\cdot]_\Pi$ associated with any skew algebroid bracket. Actually, a skew algebroid is a skew-symmetric bilinear operation which is a quasi-derivation with respect to each argument.

Note that a quasi-derivation $D$ on sections of $A$ induces canonically the {\it dual quasi-derivation} $D^*$ on sections of $A^\ast$, with the same anchor $\wh{D}\in\X(M)$, defined {\it via}
\be\label{dqd} \bk{D^\ast(\za)}{Y}=\wh{D}\left(\bk{\za}{Y}\right)-\bk{\za}{D(Y)}\,.
\ee
It is easy to see that linear vector fields on $A$ are projectable and correspond to quasi-derivations on sections of $A^\ast$.
\begin{theorem} There is a one-one correspondence between quasi-derivations $D$ on sections of $A^\ast$ and linear vector fields $\zq^D$ on a vector bundle $A$ over $M$ established by the formula
\be\label{eph} \zq^D(\zi(\za))=\zi(D(\za))\,.
\ee
The vector field $\zq^D$ is projectable onto the anchor $\wh{D}\in\X(M)$.
\end{theorem}
The following theorem is an abstract version of the complete tangent lift for sections of a skew algebroids as defined in \cite{GU1,GU3,GU}.

\begin{theorem} With any quasi-derivation $D$ on a vector bundle $A$ over $M$ there is associated a unique linear vector field $D^c$ on $A$, the {\it complete lift of} $D$, defined as $D^c=\zq^{D^\ast}$.
In local coordinates $(x^a,z^i)$ associated with a local basis $f_i$ of sections of $A$, the complete lift of the quasi-derivation $$D(h^i(x)f_i)=\left(h^i(x)b^k_i(x)+\frac{\pa h^k}{\pa x^a}\zs^a(x)\right)f_k$$
reads as
\be\label{cl} D^c(x,z)=\zs^a(x)\pa_{x^a}-z^ib^k_i(x)\pa_{z^k}\,.
\ee
\end{theorem}
Note that $\dd_T^\Pi(X)$ of a section $X$ of a skew algebroid is defined as $(ad_X^\Pi)^c$ and, like in \cite{GU1,GU3,GU}),
the vector field $D^c$ can be characterized also as the unique vector field on $A$ such that, for each section $Y$ of $A$, the vertical part of $D^c$ with respect to the decomposition of \ $T A_{|Y(M)}$ into vertical vectors and vectors tangent to the image $Y(M)$ of $Y$ is the negative of the vertical lift, $-D(Y)^v$, of the section $D(Y)$:
\be\label{vii}\forall\ x\in M\ \left(D^c+D(Y)^v\right)(Y(x))\in T_{Y(x)}Y(M)\,.
\ee
Since the complete lifts are linear vector fields, with every time dependent quasi-derivation $D=\{ t\mapsto D_t\}$ we can associate a (local) one-parameter family $t\mapsto \exp_t(D)$ of automorphisms of the vector bundle $A$ such that
$$\left.\frac{\xd}{\xd t}\right| _{t=t_0}\exp_t(D)(v)=D^c_{t_0}(v)\,,\quad \exp_0(D)=id_A\,.$$
The trajectories $t\mapsto \exp_t(D)(v)$ of the time-dependent vector field $t\mapsto D^c_{t}$ on $A$ project onto trajectories of the time-dependent vector field $t\mapsto \wh{D_t}$ on $M$.

\bigskip
For connections on line bundles, exactly like in the case of a Lie algebroid, we have the following.
\begin{theorem} Let $\nabla$ be an $E$-connection on a line bundle $L$ and let $\zs$ be a nowhere vanishing section of $L$.
Then, there is a section $\ph^\nabla_\zs$ of $E^*$ (called the {\sl characteristic form of $\zs$}) such that,
for all $X\in\Sec(E)$,
\be\label{cf}\nabla_X\zs=\langle X,\ph^\nabla_\zs\rangle\zs\,.\ee
Moreover, if $\zs'=e^f\cdot\zs$ for some smooth function $f\in C^\infty(M)$, then
$\ph^\nabla_{\zs'}=\ph^\nabla_\zs+\xdp f$.
\end{theorem}
In view of the above theorem, the class of $\ph^\nabla_\zs$ in in the quotient space $[\A^1(E^*)]=\A^1(E^*)/\xdp\A^0(E^*)$ does not depend on
the choice of the section $\zs$ trivializing the bundle $L$. We call it the {\it characteristic class} of the
connection $\nabla$ and denote, with some abuse of notation, $\ch(L)$. If $E$ is a Lie algebroid and $\nabla$ is a representation, $\nabla_X\nabla_{X'}-\nabla_{X'}\nabla_{X}=\nabla_{[X,X']_\zP}$, then the characteristic form is closed, $\xdp\ph^\nabla_\zs=0$, so $\ch(L)\in H^1(E)$.

Note that if $\nabla_i$ is an $E$-connection on a line bundle $L_i$ over $M$, $i=1,2$, then $\nabla=\nabla_1\ot Id+Id\ot\nabla_2$ is an $E$-connection on $L_1\ot L_2$ and $\ch(L_1\ot L_2)=\ch(L_1)+\ch(L_2)$. This allows us to define $\ch(L)$ even
when $L$ is not trivializable: as $L\ot L$ is trivializable, we put $\ch(L)=\frac{1}{2}\ch(L\ot L)$. We can
also use sections $|\zs|$ of the `absolute value' bundle $|L|=L/\Z_2$ (cf. \cite{GMM}). Since
$\ph^\nabla_\zs=\ph^\nabla_{-\zs}=\ph^\nabla_{|\zs|}$, and since there is a nowhere-vanishing section $|\zs_0|$
of $|L|$, we can define $\ch(L)$ as the class of $\ph(|\zs_0|)$, where
\be\label{cf0}\nabla_X\zs_0=\langle X,\ph^\nabla_{|\zs|}\rangle\zs_0\ee
for any local representative $\zs_0$  of $|\zs_0|$.

\section{Modular class of a skew algebroid}
\noindent For a skew algebroid $(E,\zP)$ consider the line bundle
$$L^E=\wedge^{\text{top}}E\otimes\wedge^{\text{top}}T^*M\,.$$
An important observation, analogous to the one in \cite{ELW}, is that there is a canonical line bundle
connection $\nabla^\zP$ on $L^E$:
\be\label{LDe}
\nabla^\zP_X(Y\ot\zm)=\cL_X(Y)\ot\zm+Y\ot\Ll_{\zr(X)}\zm\,,
\ee
where $\Ll$ is the standard Lie derivative along a vector field. The corresponding characteristic class
$\ch(L^E)\in[\A^1(E^\ast)]$ we call the {\it modular class} of the skew algebroid $E$ and denote $\mod(E,\Pi)$
(shortly, $\mod(E)$ if $\Pi$ is is fixed). We call the skew algebroid {\it modular} if its modular class
vanishes. The direct check shows that, for standard affine local coordinates $(x^a,y^i)$ in $E$ and for
the standard local section $\zs$ associated with these local coordinates,
\be\label{sec}\zs=e_1\we\dots\we e_n\otimes\xd
x^1\we\dots\we\xd x^m\,,
\ee
the modular form $\ph^\Pi_\zs:=\ph^{\nabla^\Pi}_\zs$ associated with $\Pi$ as in (\ref{SA}) reads
\begin{equation}\label{aa}
\ph^\Pi_\zs=\Bigl(\sum_k c^k_{ik}(x)+\sum_a\frac{\pa\zr^a_i}{\pa x^a}(x)\Bigr)e^i.
\end{equation}
Exactly like in the case of a Lie algebroid, we get easily the following.
\begin{theorem}
A skew algebroid $(E,\Pi)$ is modular if and only if there is a nowhere-vanishing section $|\zs|$ of $|L^E|$
such that $\cL_X|\zs|=0$ for all $X\in\Sec(E)$.
\end{theorem}
\noindent We understand, clearly, that $\cL_X|\zs|$ is represented locally by $\cL_X\zs$ for any local sections $\zs$ of $L^E$ representing $|\zs|$. Note that with any nowhere-vanishing section $\zs$ of $L^E$ we can canonically associate a volume
form (a nowhere-vanishing top-rank form) $\wt{\zs}$ on $E^*$. In local coordinates,
$${\left(f(x)e_1\we\dots\we e_n\otimes\xd
x^1\we\dots\we\xd x^m\right)}^{\thicksim}=f(x)\xd\zx_1\we\dots\we \xd\zx_n\we\xd x^1\we\dots\we\xd x^m\,.$$
Similarly, with a nowhere-vanishing section $|\zs|$ of $|L^E|$ we associate a nowhere-vanishing density
$\wt{|\zs|}$.

The volume form $\wt{\zs}$ (density $\wt{|\zs|}$) is {\it homogeneous}. More precisely, it is homogeneous of
degree $n=\on{rank}E$ with respect to the Euler (Liouville) vector field $\zD=\zx_i\pa_{\zx^i}$ on $E^*$:
$$\Ll_\zD\wt{\zs}=n\wt{\zs}
$$
(for all $\zs$ representing $|\zs|$).
\begin{proposition} Any homogeneous volume form (density) on the dual $E^*$ of a skew algebroid $E$ equals
$\wt{\zs}$ (resp., $\wt{|\zs|}$) for a nowhere-vanishing section $\zs$ of $L^E$ (resp., a nowhere-vanishing
section $|\zs|$ of $|L^E$).
\end{proposition}
\begin{proof}
We can assume that $L^E$ is trivializable, so there is a nowhere-vanishing section $\zs_1$ of $L^E$. If $\zW$
is a volume form, $\zW$ (or $-\zW$, does not matter) can be written as $\zW=e^\zc\cdot\wt{\zs_1}$ for some
$\zc\in C^\infty(E^*)$. Now, if $\zW$ is homogeneous, then
$$n\zW=\Ll_\zD\zW=\Ll_\zD(e^\zc\cdot\wt{\zs_1})=e^\zc\left(\zD(\zc)\wt{\zs_1}+n\wt{\zs_1}\right)=
\zD(\zc)\zW+n\zW\,,$$ so $\zD(\zc)=0$. Thus, $\zc$ is homogeneous of degree 0, so it is constant along fibres and
therefore is the pull-back $\zc=f\circ\zp$ of a function $f$ on the base. Hence, $\zW=\wt{e^f\zs_1}$ and the section
$\zs=e^f\zs_1$ is nowhere-vanishing.
\end{proof}

Let $\left(\ph^\Pi_{|\zs|}\right)^v$ be the vertical vector field on $E^*$ being the vertical lift of the
modular form associated with $|\zs|$. Let $\on{div}_{\wt{|\zs|}}$ denotes the divergence associated with
the density $\wt{|\zs|}$, i.e.
$$\Ll_\X\wt{|\zs|}=\on{div}_{\wt{|\zs|}}(\X)\wt{|\zs|}\,.$$
One can easily prove by direct calculations the following (cf. \cite{We}).
\begin{theorem}\label{mt0} For each Hamiltonian $H\in C^\infty(E^*)$,
$$\left(\ph^\Pi_{|\zs|}\right)^v(H)=\on{div}_{\wt{|\zs|}}\X_H\,.$$
\end{theorem}

Among all Hamiltonians we can distinguish those which are of {\it mechanical type}: the kinetic energy $+$ a
potential function. The kinetic energy is associated with a bundle (fiberwise) metric, i.e. a smooth family of scalar products
on fibers of $E^\ast$ (or $E$). Choosing a local basis of orthonormal sections, we can write in the
corresponding local coordinates
\be\label{mt} H(x,\zx)=\frac{1}{2}\sum_i\zx_i^2+V(x)\,.
\ee
Since, for any such Hamiltonian,
\be\label{Hmt}\left(\ph^\Pi_{|\zs|}\right)^v(H)=\Bigl(\sum_k c^k_{ik}(x)+\sum_a\frac{\pa\zr^a_i}{\pa x^a}(x)\Bigr)\frac{\pa H}{\pa \zx_i}=\Bigl(\sum_k c^k_{ik}(x)+\sum_a\frac{\pa\zr^a_i}{\pa x^a}(x)\Bigr)\zx_i\,,
\ee
the modular form $\ph^\Pi_{|\zs|}$ vanishes if and only if $\left(\ph^\Pi_{|\zs|}\right)^v(H)=0$ for a single
mechanical Hamiltonian. We thus get the following (cf. \cite{M,We}).
\begin{theorem}\label{H} Let $E$ be a skew algebroid. The following are equivalent:
\begin{itemize}
\item[(a)] $E$ is unimodular.
\item[(b)] There is a homogeneous density on $E^\ast$ which is invariant with respect to all Hamiltonian vector fields.
\item[(c)] There is a homogeneous density on $E^\ast$ which is invariant with respect to a certain hamiltonian vector field associated with a Hamiltonian of mechanical type.
\end{itemize}
\end{theorem}
\begin{proof} $(a) \Leftrightarrow (b)$ \
We can assume that $L^E$ is trivializable and work with forms instead of densities. It is clear from
Theorem \ref{mt0} that if the skew-algebroid is modular, then there is a homogeneous volume form on $E^*$ invariant with
respect to all Hamiltonian vector fields. This is $\wt{\zs}$ for a nowhere-vanishing section of $L^E$ whose
modular form is 0.

We have also the converse. Indeed, if $\zW$ is a homogeneous volume form, we know that it can be written as
$\zW=\wt{\zs}$ for some nowhere-vanishing section $\zs$ of $L^E$. If $\zW$ is invariant with respect to all
the Hamiltonian vector fields $\X_H$, then $0=\on{div}_{\wt{\zs}}\X_H=\left(\ph^\Pi_\zs\right)^v(H)$ for all
Hamiltonians $H$. This means $\left(\ph^\Pi_\zs\right)^v=0$, thus $\ph^\Pi_\zs=0$, so the skew algebroid is
modular. Finally, $(a) \Leftrightarrow (c)$ follows immediately from (\ref{Hmt}).
\end{proof}

It is {\it a priori} possible that a non-modular skew algebroid admits a non-homoge\-neous volume form $\zW$
which is invariant with respect to all Hamiltonian vector fields. Since $\zW=e^\zc\wt{\zs}$ for some
nowhere-vanishing section $\zs$ of $L^E$, the property that $\zW$ is invariant with respect to all $\X_H$
reads
$$0=\on{div}_\zW(\X_H)=\on{div}_{\wt{\zs}}(\X_H)+\X_H(\zc)=
\left(\left(\ph^\Pi_\zs\right)^v-\X_\zc\right)(H)\,.$$ As $H$ is arbitrary,
$\left(\ph^\Pi_\zs\right)^v=\X_\zc$. If $\zc$ is a basic function, this implies that $\ph^\Pi_\zs$ is exact
and the algebroid is modular. It is, however, not excluded that $\left(\ph^\Pi_\zs\right)^v=\X_\zc$ for some
function $\zc$ which is not basic. Note that the vertical lift of $\xdp f$, for a basic functions $f$,
is automatically a Hamiltonian vector field, $\left(\xdp f\right)^v=-\X_f$. In this way we get the following.
\begin{theorem} There is a density on $E^*$ which is invariant
with respect to all Hamiltonian vector fields if and only if the vertical lift of a (thus any) modular form is a
Hamiltonian vector field.
\end{theorem}

\begin{remark} Skew algebroids appear naturally in nonholonomic mechanics and the modular class in this context has been studied also in \cite{FGM}. Note also that there is an extensive literature devoted to mechanics on Lie algebroids, the program proposed by Weinstein \cite{We1} and developed by Mart\'\i nez and others (see the survey article \cite{LMM}). In \cite{GGU} and \cite{GG}, in turn, it was developed  a framework of geometric mechanics on {\it general algebroids}. The structure of an algebroid on a bundle $\tau:E\ra M$ allows one to develop
the Lagrangian formalism for a given Lagrangian function $L:E\ra\R$.
Moreover, if $E$ is an almost Lie algebroid, then the associated
Euler-Lagrange equations have a variational interpretation: a
curve $\gamma:[t_0,t_1]\ra E$ satisfies the Euler-Lagrange
equations if and only if it is an extremal of the action
$\mathcal{J}(\gamma):=\int_{t_0}^{t_1}L(\gamma(t))\dd t$
restricted to those $\gamma$'s which are admissible and belong to
a fixed $E$-homotopy class \cite{GG}.

In \cite{GLMM} it has been shown that if $D\subset E$ is a
subbundle (nonholonomic constraint) and $L$ is of mechanical type (that is $L(x,v)=\frac 12
\mu(v,v)-V(x)$, where $\mu$ is a bundle metric on $E$ and $V$ is an
arbitrary function on the base), then nonholonomically constrained
Euler-Lagrange equations associated with $D$ can be obtained as
unconstrained Euler-Lagrange equations on the skew-algebroid
$\left(D,\rho_{|D},[\cdot,\cdot]_\Pi^D:=\PP_D[\cdot,\cdot]_\Pi\right)$,
where $\PP_D:E\ra D$ denotes the projection orthogonal w.r.t. $\mu$. In other words, the skew algebroid structure is obtained simply by projecting the tensor $\Pi$ on $D$ by means of $TD$. Note that the skew algebroid bracket
$[\cdot,\cdot]^D_\Pi$ needs not to satisfy Jacobi identity even if
$[\cdot,\cdot]_E$ does.
\end{remark}
\begin{example}
To give a concrete example let us study the Chaplygin
sleigh. It is an example of a nonholonomic system on the Lie
algebra $\mathfrak{se}(2)$ which describes a rigid body sliding on
a plane. The body is supported in three points, two of which
slides freely without friction, while the third point is a knife
edge. This imposes the constraint of no motion orthogonal to this
edge (see \cite{Cha}).

The configuration space before reduction is the Lie group
$G=SE(2)$ of the Euclidean motions of the 2-dimensional plane
$\R^2$. Elements of the Lie algebra $\mathfrak{se}(2)$ are of the
form

$$\hat{\xi}=
\begin{pmatrix}
0&\omega&v_1\\
-\omega&0&v_2\\
0&0&0
\end{pmatrix}=v_1E_1+v_2E_2+\omega E_3,
$$
where $[E_3,E_1]=E_2$, $[E_2, E_3]=E_1$, and  $[E_1, E_2]=0$.

The system is described by the purely kinetic Lagrangian function
$L:\mathfrak{se}(2)\ra\R$,
$$L(v_1, v_2, \omega)=\frac{1}{2}\left[ (J+m(a^2+b^2))\omega^2 + mv_1^2+m v_2^2-2bm\omega v_1-2am\omega v_2\right].$$
Here,  $m$ and $J$ denote the mass and the moment of inertia of
the sleigh relative to the contact point, while $(a, b)$
represents the position of the center of mass with respect to the
body frame, determined by placing the origin at the contact point
and the first coordinate axis in the direction of the knife axis.
Of course, as the Lie algebra $\mathfrak{se}(2)$ is modular, the free dynamics admits an invariant volume for the corresponding Hamiltonian. However, our  system is subjected to the nonholonomic
constraint determined by the linear subspace
$$ D=\{(v_1, v_2, \omega)\in \mathfrak{se}(2)\; |\; v_2=0\}\subset\mathfrak{se}(2).$$
Instead of $\{E_1, E_2, E_3\}$ we take another basis of
$\mathfrak{se}(2)$:
$$
e_1=E_3, e_2=E_1, e_3= -ma E_3-mab E_1+(J+ma^2) E_2\,,$$ adapted
to the decomposition $D\oplus D^\perp$; $D=\hbox{span }\{ e_1,
e_2\}$ and $D^\perp=\hbox{span }\{ e_3\}$. The induced
skew-algebroid structure on $D$ is given by
$$
[e_1, e_2]_{D}=\frac{ma}{J+ma^2} e_1+\frac{mab}{J+ma^2}e_2.$$
Therefore, the structural constant are ${\mathcal C}^1_{12}=\frac{ma}{J+ma^2}$ and
${\mathcal C}^2_{12}=\frac{mab}{J+ma^2}$. The skew algebroid $D$ is
in this simple case a Lie algebra, as any skew algebra in dimension 2 is Lie.
According to (\ref{aa}),
$$\mod(D)=\frac{mab}{J+ma^2}e^1-\frac{ma}{J+ma^2}e^2\,,$$
so $\mod(D)\ne 0$ and the associated Hamiltonian does not admit an invariant volume.
\end{example}

\section{Supergeometric description}
Our aim now is to express the modular class $\mod(E)$ directly in terms of the degree 1 vector field $\xdp$ on
the graded manifold $E[1]$. To this end the sheaf of sections of $L^E$ has to be replaced by the {\it
Berezinian sheaf} \ $\mathrm{Ber}=\mathrm{Ber}(E[1])$ on the supermanifold $E[1]$ whose nowhere-vanishing
sections are {\it Berezinian volumes} (cf. \cite{HM1,HM2}). A homogeneous Berezinian volume $s$ defines a
divergence $\div_s$ of a homogeneous vector field $X$ by the formula (see \cite{KSM})
$$ \cL_Xs=(-1)^{|X||s|}s\cdot\div_s(X)\,.$$
Here $\cL_Xs$ is the Lie derivative of the Berezinian volume understood as a differential operator on $\A$
defined by $\cL_Xs=-(-1)^{|X||s|}s\circ X$.
\medskip
Over a superdomain $\U$ with supercoordinates $\uu=(x^1,\ldots,x^m,\y^1,\ldots,\y^n)$ the (right)
$\A(\U)$-module $\mathrm{Ber}(\U)$ is generated by
$$s=d^{\,m|n}\uu=dx^1\wedge\ldots\wedge dx^m\otimes \pa_{\y^n}\circ\ldots\circ\pa_{\y^1}\,,$$
$$\mathrm{Ber}(\U)=\Sec(\U,\mathrm{Ber}):=d^{\,m|n}\uu\cdot\A(\U)\,,
$$
and the corresponding divergence of a homogeneous vector field $X=\sum_ag_a\pa_{x^a}+\sum_ih_i\pa_{\y^i}$
reads (cf. \cite{KSM})
\be\label{div}\div_s(X)=\sum_a\frac{\pa g_a}{\pa x^a}-(-1)^{|X|}\sum_i\frac{\pa h_i}{\pa\y^i}\,,
\ee
where $|X|$ is the degree of $X$.

\medskip
It is easy to see that with any section $\zs$ of the line bundle
$L^E=\wedge^{\text{n}}E\otimes\wedge^{\text{m}}T^*M$ one can associate a Berezinian section $s_\zs$. In other
words, we have an embedding $S$ of the space of sections of $L^E$ into the space of Berezinian sections. In
local coordinates this embedding reads
$$S:\quad f(x)\,e_n\wedge\ldots\wedge e_1\otimes \dd x^1\wedge\ldots\wedge \dd x^m\ \mapsto \ f(x)\,\dd x^1\wedge\ldots\wedge \dd x^m\otimes
\pa_{\y^n}\circ\ldots\circ\pa_{\y^1}\,,$$ where $e_1,\ldots,e_q$ is a basis of local sections of $E$ and
$\pa_{\y^i}$ is the derivation of $\A(E^\ast)$ being the contraction of a section of $\wedge^\bullet E^*$ with
$e_i$.

Recall that for any nowhere-vanishing section $|\zs|$ of the bundle $|L^E|$ we can find an open covering $\{ U_\alpha\}$ on $M$
and local nowhere-vanishing sections $\zs_\alpha$ of $L^E_{|U_\alpha}$ such that $\zs_\alpha=\pm
\zs_{\alpha'}$ on $U_\alpha\bigcap U_{\alpha'}$. Hence, $S(\zs_\alpha)=\pm S(\zs_{\alpha'})$ over
$U_\alpha\bigcap U_{\alpha'}$. But the divergence $\div_s$ does not depend on the sign of $s$, so the
collection $|s|=S(|\sigma|)$ of local nowhere-vanishing Berezinian volumes $S(\zs_\alpha)$ gives rise to a
super(divergence) $\div_{|s|}$ on $E[1]$. Note that $|s|$ can be viewed as a nowhere-vanishing section of the
sheaf \ ${\frak D}$ of {\it Berezinian $1$-densities} defined as the Berezinian sheaf twisted by on orientation
sheaf. In the literature the sections of ${\frak D}$ are sometimes referred to as {\it non-oriented
Berezinian sections}. It is easy to see that if $|s'|=f|s|$,  with an even nowhere-vanishing function $f$, is
another nowhere-vanishing Berezinian 1-density, then $\div_{|s'|}(X)=\div_{|s|}(X)+X(\ln(f))$.

If we assume now that $(E,\Pi)$ is a skew algebroid and $\xdp$ is the corresponding vector field on $E[1]$
(see (\ref{sdR})), then, according to (\ref{div}), it is easy to see in local coordinates that, for the
standard local section $\zs$ of $L^E$ associated with local coordinates (see (\ref{sec})),
\be\label{sdiv}\div_{S(\zs)}(\xdp)=\Bigl(\sum_k c^k_{ik}(x)+\sum_a\frac{\pa\zr^a_i}{\pa x^a}(x)\Bigr)\y^i\,.
\ee
The latter superfunction represents clearly the local section $\ph^\Pi_\zs$ of $E^\ast$. Moreover, if $\zs'=e^f\zs$
for some function $f\in C^\infty(M)$, then $S(\zs')=e^fS(\zs)$ and
$$\div_{S(\zs')}(\xdp)=\div_{S(\zs)}(\xdp)+\xdp f$$
that represents the section $\ph^\Pi_\zs+\xdp f$ of $E^\ast$. All this works well globally for 1-densities and
we get the following.
\begin{theorem} Let $(E,\Pi)$ be a skew algebroid and $\xdp$ be the de Rham vector field on the supermanifold $E[1]$ corresponding to $\Pi$. Then,
the modular class $\mod(E)$ is represented by the degree-1 function $\div_{|s|}(\xdp)$ being the
(super)divergence of $\xdp$ with respect to any homogeneous nowhere-vanishing Berezinian 1-density $|s|$ on $E[1]$
(associated canonically with a nowhere-vanishing section $|\zs|$ of the line-bundle $|L^E|$).
\end{theorem}
\begin{remark} We found that the above interpretation of the modular class in the case of Lie algebroids have been indicated by Kontsevich \cite{Kn}. Note also that superdivergences have been used in constructing generating operators for Lie algebroids in \cite{KSM}, and in defining modular class for even symplectic graded manifolds in \cite{MV}.
\end{remark}

\section{Relative modular class}
Let now $E_0$ be a skew subalgebroid in our skew algebroid $(E,[\cdot,\cdot]_\Pi,\zr)$. This means that
(see \cite{GG1,Ma}, although it seems that the third condition in \cite[Definition 4.3.14]{Ma}
is superfluous):
\begin{itemize}
\item[(a)] $E_0$ is a vector subbundle in $E$ supported on a submanifold $M_0$ of $M$; \item[(b)] The anchor
$\zr:E\ra T M$ maps $E_0$ into $T M_0\subset T M$, so that the bracket $[\cdot,\cdot]_\Pi$ is well defined
for sections of $E_0$, as its value over $M_0$ does not depend on extensions of these sections; \item[(c)] Sections of $E_0$ are closed with respect to the bracket $[\cdot,\cdot]_\Pi$.
\end{itemize}

It is easy to see that the restriction of the bracket $[\cdot,\cdot]_\Pi$ to sections of $E_0$ is a skew
algebroid bracket associated with a certain bivector field $\Pi_0$ on $E_0^\ast$. In other words, the above
conditions mean that the tensor $\Pi$ is tangent to the submanifold $E^\ast_{|M_0}$ and projectable onto a
linear bivector field $\Pi_0$ on $E_0^\ast\simeq E^\ast_{|M_0}/E_0^\perp$, where $E_0^\perp\subset
E^\ast_{|M_0}$ is the annihilator of $E_0$. The association $\Pi\mapsto\Pi_0$ is a simple example of a Poisson
reduction. It is also clear that $E_0$ inherits from $E$ a structure of a skew algebroid that is associated
with $\Pi_0$.

All this can be immediately seen in local coordinates in terms of the structure functions $\zr(x)$ and $c(x)$. We can
choose affine coordinates
$$(x^\za,x^A,y^\zi,y^I)=(x^1,\dots,x^{m_0},x^{m_0+1},\dots,x^m,y^1,\dots,y^{n_0},y^{n_0+1},\dots,y^n)$$
in a neighbourhood of a fiber of $E_0$, so that  points of $M_0$ in $M$ are characterized by $x^A=0$, and
points of $E_0$ in $E$ by $x^A=0,y^I=0$. The fact that $E_0$ is a subalgebroid means that $\zr_\zi^A(x^a,0)=0$
and $c_{\zi\zi'}^K(x^a,0)=0$, so that
$$\Pi_0(x^\za)=c^\zk_{\zi\zi'}(x^\za,0)\zx_\zk
\partial _{\zx_\zi}\otimes \partial _{\zx_{\zi'}} + \zr^\zb_\zi(x^\za,0) \partial _{\zx_\zi}\we\pa_{x^\zb}\,.
$$
It is also easy to see that subbundles of $E$ (supported over submanifolds) are graded
submanifolds of $E[1]$, so that we have the following equivalent descriptions of a skew subalgebroid.
\begin{theorem}\label{th1} Let $E_0$ be a subbundle, supported on $M_0\subset M$, of a skew algebroid $(E,\Pi)$ over $M$ (equivalently, a graded submanifold $E_0[1]$ of the N-manifold $\cM=E[1]$). The following are equivalent:
\begin{itemize}
\item[(a)] $E_0$ is a subalgebroid with a skew algebroid structure induced by a bivector $\Pi_0$ on
$E_0^\ast$.

\item[(b)] The annihilator $E_0^\perp$ is a coisotropic submanifold of the Poisson manifold $(E^\ast,\Pi)$.

\item[(c)] The (super)vector field $\xdp$ is tangent to the submanifold $E_0[1]$ of $E[1]$ and $\xd^{\Pi_0}$
is the restriction $\xdp_{|E_0[1]}$.

\item[(d)] the canonical inclusion $j_{E_0}:E_0\hookrightarrow E$ is a morphism of skew algebroids, i.e., the
induced restriction map $j_{E_0}^\ast:\A(E^\ast)\ra\A(E_0^\ast)$ intertwines the derivations $\xdp$ and
$\xd^{\Pi_0}$,
\be\label{sam} j_{E_0}^\ast\circ\xdp=\xd^{\Pi_0}\circ\ j_{E_0}^\ast\,.
\ee
\end{itemize}
\end{theorem}
\begin{proof} (a) $\Rightarrow$ (b) The ideal $J$ of function vanishing on $E_0^\perp$ is generated by basic
functions vanishing on $M_0$ and sections of $E$ extending sections of $E_0$. The equations (\ref{drel}) and
(\ref{drel1}) immediately imply that $\{ J,J\}_\Pi=0$.

(b) $\Rightarrow$ (c) As the ideal $\mathbf{J}$ of functions on $E[1]$ vanishing on $E_0[1]$ is generated by basic
functions vanishing on $M_0$ and functions associated with extensions of sections of $E_0^\perp$, so locally $x^A$ and $\mathbf{y}^I$, we get
$\xdp(\mathbf{J})\subset \mathbf{J}$ in view of (\ref{r3}), (\ref{r4}), and the fact that $\zr_\zi^A(x^a,0)=0$ and
$c_{\zi\zi'}^K(x^a,0)=0$.

(c) $\Rightarrow$ (d) is obvious, and (d) implies that $\zr_\zi^A(x^a,0)=0$ and $c_{\zi\zi'}^K(x^a,0)=0$, so
(a).

\end{proof}
\noindent Note that (\ref{sam}) implies that $j_{E_0}^\ast$ induces a map on classes,
$$j_{E_0}^\ast:[\A^1(E^*)]\ra[\A^1(E_0^*)]$$
($j_{E_0}^\ast:H^1(E)\ra H^1(E_0)$ in the case of a Lie algebroid). Let $\mod(E)$ (resp, $\mod(E_0)$) be
the modular class of the skew algebroid $(E,[\cdot,\cdot]_\Pi)$ (resp., $(E_0,[\cdot,\cdot]_{\Pi_0})$). It is
easy to see that, in general,
$$ \mod(E_0)\ne j_{E_0}^\ast(\mod(E))\,.
$$
Actually, as easily checked in local coordinates as above and the corresponding sections of $L^E$,
\bea\nonumber j_{E_0}^\ast(\ph^\Pi_\zs)(x^\za)&=&
\sum_{\zi=1}^{n_0}\Bigl(\sum_{\zk=1}^{n_0} c^\zk_{\zi\zk}(x^\za,0)+\sum_{K=n_0+1}^nc^K_{\zi K}(x^\za,0)+
\sum_{\zb=1}^{m_0}\frac{\pa\zr^\zb_\zi}{\pa
x^\zb}(x^\za,0)\Bigr)e^\zi\\&=&\ph^{\Pi_0}_\zs(x^\za)+\sum_{\zi=1}^{n_0}\sum_{K=n_0+1}^nc^K_{\zi K}(x^\za,0)e^\zi\,.
\label{aa1}\eea The class
$$\mod(E_0;E)=\mod(E_0)-j_{E_0}^\ast(\mod(E))\,,$$
represented by
\be\label{rc}\ph^{\Pi_0}_\zs(x^\za)-j_{E_0}^\ast(\ph^\Pi_\zs)(x^\za)=-\sum_Kc^K_{\zi K}(x^\za,0)e^\zi\,,\ee
we will call the {\it relative modular class} of the
subalgebroid $E_0$ in $E$.

\medskip
It is easy to see that there is a canonical $E_0$-connection $\nabla^{\zn(E_0)}$ in the normal bundle
$\zn(E_0)=E_{|M_0}/E_0$ ({\it linear holonomy connection}):
\be\label{norc}\nabla^{\zn(E_0)}_X([e])=[[X,e]_\zP]\,,
\ee
(where $[e]$ is the class of a section $e$ of $E$ in $\zn(E_0)$), thus a canonical connection
$\nabla^{\zn(E_0)}$ in the line bundle $\we^{top}\zn(E_0)$.
In particular, for $\zi=1,\dots,n_0$ and $J=n_0+1,\dots, n$,
\be\label{connection}
\nabla^{\zn(E_0)}_{e_{\zi}}([e_J])=\sum_{K=n_0+1}^nc^K_{\zi J}(x^\za,0)[e_K]
\ee
and, for a section
$X=f^\zi(x^1,\dots,x^{m_0})e_\zi$ of $E_0$,
$$\nabla^{\zn(E_0)}_X\left([e_{n_0+1}]\we\cdots\we [e_n]\right)=
\sum_{\zi=1}^{n_0}\sum_{K=n_0+1}^nc^K_{\zi K}(x^\za,0)f^\zi(x^\za)[e_{n_0+1}]\we\cdots\we [e_n]\,.
$$
This, compared with (\ref{rc}), proves the following.
\begin{theorem} The relative class $\mod(E_0;E)$ is the obstruction to the existence of a transversal to $E_0$
measure invariant under the linear holonomy:
\be\label{norc1} \mod(E_0;E)=-\ch(\we^{top}{\zn(E_0)})=\ch(\we^{top}{\zn^\ast(E_0)})\,.
\ee
\end{theorem}

\section{Holonomy}

Let now $\zg:I=[0,1]\to E$ be a smooth {\it admissible path} in the skew algebroid $E$, i.e.,
\be\label{admissible}
\forall \ t\in I\quad \zr(\zg(t))=\dot{\ulg}(t)\,,
\ee
where $\ulg:I\ra M$ is the projection of $\zg$ to $M$, $\ulg=\zt\circ\zg$, and $\dot{\ulg}:I\ra T M$ is the tangent prolongation of $\ulg$.

Fixing an $E$-connection $\nabla$ on a vector bundle $A$ over $M$, we can define the {\it parallel transport of the fibers of $A$ of along $\zg$} (cf. \cite{CF1,Fe}) as the path $[0,1]\ni t\mapsto F_t$, $F_t:A_{\ulg(0)}\ra A_{\ulg(t)}$, of fiber morphisms associated with solutions of the differential equation
\be\label{pt}
\nabla_{\zg(t)}Y(t)=0\,,\quad Y(t)\in A_{\ulg(t)} \,,
\ee
in an obvious way.
In local coordinates associated with a choice of a local basis $f_1,\dots,f_s$ of sections of $A$, (\ref{pt}) reads as
\be\label{pt0}
\dot{Y}^j(t)+\zG^j_{ik}(\ulg(t))\zg^i(t)Y^k(t)=0\,,
\ee
where
$$\nabla_{e_i}(u^k(x)f_k)=\left(\zr^a_i(x)\frac{\pa u^j}{\pa x^a}(x)+\zG_{ik}^j(x)u^k(x)\right)f_j\,.
$$
Note that
$$\zg^i(t)\zr^a_i(\ulg(t))\frac{\pa Y^j}{\pa x^a}(\ulg(t))=\dot{Y}^j(t)$$
due to admissibility of $\zg$, independently on the extension of $Y$ to a time-dependent section of $A$.
Like in the case of a Lie algebroid, one can easily prove that the parallel transport along $\zg$ can be described as the flow $\exp_t(D)$ over $\ulg$,
\be\label{pt1}
F_t(v)=\exp_t(D)(v)\,,\quad v\in A_{\ulg(0)}\,,
\ee
where $D_t=\nabla_{\zg(t)}$. Of course, since formally $\nabla_{\zg(t)}$ is not a globally defined quasi-derivative, one takes $\nabla_{X_t}$ instead, where $X=\{ t\mapsto X_t\}$ is any time-dependent section of $E$ such that $X_t(\ulg(t))=\zg(t)$; the r.h.s of (\ref{pt1}) does not depend on the choice of $X$.

Let
$\zw=\xd Y^1\we\cdots\we\xd Y^s$ be a homogeneous volume form on fibers associated with the chosen local basis of sections. In view of (\ref{pt0}), the determinant of the parallel transport, $h(t)=\det(F_t)$, expressed in these coordinates satisfies the differential equation
\be\label{deq}
\dot{h}(t)+\sum_{i,j}\zG^j_{ij}(\ulg(t))\zg^i(t)h(t)=0\,,
\ee
thus
\be\label{deq1}
h(1)=\exp\left(-\int_0^1\left(\sum_{i,j}\zG^j_{ij}(\ulg(t))\zg^i(t)\right)\xd t\right)\,.
\ee
All this, applied to the liner holonomy connection $\nabla^{\zn(E_0)}$ in the normal bundle
$A=\zn(E_0)=E_{|M_0}/E_0$ associated with a subalgebroid $E_0$ of $E$, gives (see (\ref{connection}))
\bea\label{holonomy} h(1)&=&\exp\left(\int_0^1\left(\sum_{\zi=1}^{n_0}\sum_{K=n_0+1}^nc^K_{\zi K}(\ulg(t))\zg^\zi(t)\right)\xd t\right)\\ &=&\exp\left(\int_0^1\left\langle \left(\ph^{\Pi_0}_\zs-j_{E_0}^\ast(\ph^\Pi_\zs)\right)(\ulg(t)),\zg(t)\right\rangle\xd t\right)=
\exp\left(\int_\zg\left(\ph^{\Pi_0}_\zs-j_{E_0}^\ast(\ph^\Pi_\zs)\right)\right)\,.\nonumber
\eea
Here, for $\za\in\Sec(E^\ast)$, we denote
$$\int_\zg\za=\int_0^1\left\langle \za(\ulg(t)),\zg(t)\right\rangle\xd t\,.$$
It is easy to see that, for $\za=\xd^\zP\! f$ and $\zg$ admissible, we have
$$\int_\zg\xd^\zP\! f=f(\ulg(1))-f(\ulg(0))\,.$$
In particular, if $\ulg$ is a loop, $\ulg(0)=\ulg(1)$, then $\int_\zg\za$ does no depend on the class of $\za$ in $\A^1(E^*)/\xdp\A^0(E^*)$. In this case, $\hol(\zg,E_0):=h(1)$ is well defined independently on the choice of local coordinates and called the {\it linear holonomy of $\zg$} relative to $E_0$. Also the r.h.s. of (\ref{holonomy}) is well defined independently on the choice of coordinates and can be rewritten in the form $\exp\left(\int_\zg\mod(E_0;E)\right)$. In this way we get the following generalization of the results of \cite{CF,GGo} on holonomy of Poisson manifolds and their Poisson submanifolds.
\begin{theorem} For any subalgebroid $E_0$ of a skew algebroid $E$ over $M$ and for any smooth admissible path $\zg$ in $E$ covering a loop in $M$, one has
\be\label{holonomy1}
\hol(\zg,E_0)=\exp\left(\int_\zg\mod(E_0;E)\right)\,.
\ee
\end{theorem}

\section{Modular classes of algebroid relations}
Given now two skew algebroids, $(E_i,\Pi_i)$ over the base $M_i$, $i=1,2$, we can consider the {\it direct
product skew algebroid} $E_1\times E_2$ over $M_1\ti M_2$ corresponding to the bivector field `$\Pi_1+\Pi_2$'
on $E^*_1\ti E^*_2$. Interpreting $\A(E^*_1\ti E^*_2)$ as $\A(E^*_1)\otimes\A(E^*_2)$, we can write
$$\xd^{\Pi_1\ti\Pi_2}=\xd^{\Pi_1}\ot\, 1+1\ot\xd^{\Pi_2}\,,$$ or simply, with some abuse of notation, $\xd^{\Pi_1\ti\Pi_2}=\xd^{\Pi_1}+\xd^{\Pi_2}$.

The canonical projections $p_i:E_1\ti E_2\ra E_i$, $i=1,2$, induce the maps
$p_i^\ast:\A(E_i^\ast)\ra\A(E_1^*\ti E_2^*)$ which are algebroid morphisms, i.e. they intertwine the
corresponding de Rham derivatives,
\be\label{dp} p_i^\ast\circ\xd^{\Pi_i}=(\xd^{\Pi_1}+\xd^{\Pi_2})\circ p_i^\ast\,.
\ee
One can easily prove the following.
\begin{theorem}
\be\label{t2}\mod(E_1\ti E_2)=p_1^*(\mod(E_1))+p_2^*(\mod(E_2))\,.\ee
\end{theorem}
\noindent We will write in short, with some abuse of notation,
$$\mod(E_1\ti E_2)=\mod(E_1)+\mod(E_2)\,.$$
The following is generally known in the case of Lie algebroids (cf. \cite[Theorem 10.4.9]{Ma} and \cite{BKS,We0,Va}). \begin{theorem} Let $(E_i,\Pi_i)$ be a skew algebroid over a base $M_i$, $i=1,2$,
and let $\zF:E_1\ra E_2$ be a vector bundle morphism covering a smooth map $\zf:M_1\ra M_2$. The following are
equivalent.
\begin{itemize}
\item[(i)] The graph $\graph(\zF)$ of $\zF$ is a subalgebroid in the direct product skew algebroid $E_1\ti E_2$.
\item[(ii)] The annihilator $\graph(\zF)^\perp$ is a coisotropic submanifold in the Poisson manifold $(E_1^\ast\ti E_2^\ast,\Pi_1\ti\Pi_2)$.
\item[(iii)] The super vector field $\xd^{\Pi_1}+\xd^{\Pi_2}$ is tangent to the submanifold $\graph(\zF)[1]$
of the graded manifold $(E_1\ti E_2)[1]$.
\item[(iv)] The induced map $\zF^\ast:\A(E_2)\ra\A(E_1)$ intertwines the de Rham derivatives,
$$\zF^\ast\circ\xd^{\Pi_2}=\xd^{\Pi_1}\circ\,\zF^\ast\,.$$
\end{itemize}
\end{theorem}
\begin{proof} It is clear that $\graph(\zF)$ is a vector subbundle in $E_1\ti E_2$ supported on the graph $\graph(\zf)$,
and we have (i) $\Leftrightarrow$ (ii) $\Leftrightarrow$ (iii) according to Theorem \ref{th1}. Assume now (i). The embedding
$j=Id_{E_1}\ti\zF:E_1\ra E_1\ti E_2$ is an isomorphism of the vector bundle $E_1$ onto $\graph(\zF)$ that
identifies the skew algebroid structure on $E_1$ with that on $\graph(\zF)$. We have then
\be\label{dd0}j^\ast\circ(\xd^{\Pi_1}+\xd^{\Pi_2})=\xd^{\Pi_1}\circ j^\ast\,.\ee
On the other hand, for any $\mu_i\in\A(E^\ast_i)$,
$$j^\ast(\mu_1\ot 1+1\ot\mu_2)=\mu_1+\zF^\ast(\mu_2)\,,$$
so that
\be\label{dd1}\xd^{\Pi_1}\mu_1+\zF^\ast(\xd^{\Pi_2}\mu_2)=\xd^{\Pi_1}\mu_1+\xd_{\Pi_1}(\zF^\ast(\mu_2))\ee
and (iv) follows. Conversely, (\ref{dd1}) implies (\ref{dd0}), thus (i).

\end{proof}
\begin{definition} A vector bundle morphism $\zF:E_1\ra E_2$ between skew algebroids $(E_i,\Pi_i)$, $i=1,2$,
which satisfies one (thus all) of the above conditions (i)--(iv) we call a {\it morphism of skew algebroids}.
Any subalgebroid in $E_1\ti E_2$ we call a {\it skew algebroid relation} between $E_1$ and $E_2$.
\end{definition}
It follows that skew algebroid morphisms are particular skew algebroid relations and that the latter can be
equivalently determined by coisotropic subbundles in $(E_1^\ast\ti E_2^\ast,\Pi_1\ti\Pi_2)$. For any skew
algebroid relation $\cR\subset E_1\ti E_2$ supported on a submanifold $M_0\subset M_1\ti M_2$ and equipped
with the induced skew algebroid structure associated with a tensor $\Pi$ on $\cR^*$, its relative modular
class
$$\mod(\cR;E_1\ti E_2)=\mod(\cR)-j_\cR^\ast\left(p_1^\ast(\mod(E_1))+p_2^\ast(\mod(E_2))\right)\,,$$
where $j_\cR:\cR\ra E_1\ti E_2$ is the canonical embedding, can be simply written as
$$\mod(\cR;E_1\ti E_2)=\mod(\cR)-j_{\cR,_1}^\ast(\mod(E_1))-j_{\cR,2}^\ast(\mod(E_2))\,,$$
where $j_{\cR,i}=p_i\circ j_\cR$, $i=1,2$.

\medskip
We can, however, define another relative modular class, $\mod(\cR)$, associated with the algebroid relation $\cR$
by comparing the classes of $E_1$ and $E_2$ with respect to this relation, namely
\be\label{mcar}\mod(\cR)=j_{\cR,1}^\ast(\mod(E_1))-j_{\cR,2}^\ast(\mod(E_2))\,.\ee
We will call it the {\it modular class of the skew algebroid relation} $\cR$. By obvious reasons we will call
$\za_i\in\A^1(E_i^\ast)$, $i=1,2$, {\it $\cR$-related} if
$$j_{\cR,1}^\ast(\za_1)=j_{\cR,2}^\ast(\za_2)\,.$$
One important remark is that the r.h.s., thus $\mod(\cR)$, has to be understood as an element in
$$\A^1(\cR)/\left(j_{\cR,1}^\ast(\xd^{\Pi_1}(C^\infty(M_1)))+j_{\cR,2}^\ast(\xd^{\Pi_2}(C^\infty(M_2)))\right)$$
and not as an element in $[\A^1(\cR)]=\A^1(\cR)/\xd^{\Pi}(C^\infty(M_0))$. This is because, in the latter
case, there is {\it a priori} no guarantee that we can find representatives $\za_i$ of $\mod(E_i)$, $i=1,2$,
which are $\cR$ related even when $\mod{\cR}$ is trivial in $[\A^1(\cR)]=\A^1(\cR)/\xd^{\Pi}(C^\infty(M_0))$.
Moreover, the r.h.s. of (\ref{mcar}) makes sense even when $\cR$ is not a subalgebroid and just a vector
subbundle of $E_1\ti E_2$, i.e. a {\it linear relation} between skew algebroids.

With this interpretation, the obvious characterization of modularity of the relation is the following.
\begin{proposition}
The modular class $\mod(\cR)\subset E_1\ti E_2$ vanishes if and only if there are nowhere-vanishing densities
$|\zs_i|\in\Sec(|L^{E_i}|)$, $i=1,2$, such that $\ph^{\Pi_1}_{|\zs_1|}$ and $\ph^{\Pi_2}_{|\zs_2|}$ are $\cR$-related.
\end{proposition}
Let $\zf_{\cR,i}$ be the projection of $M_0$ on $M_i$, $i=1,2$. Completely analogously like for Lie algebroids
we can consider pull-back bundles $\zf_{\cR,1}^!\left(L^{E_1}\right)$ and
$\zf_{\cR,2}^!\left((L^{E_2})^\ast\right)$ equipped with the corresponding pull-back connections
$j_{\cR,1}^!(\nabla^{\Pi_1})$ and $j_{\cR,2}^!(\left(\nabla^{\Pi_2})^*\right)$ (cf. \cite[Proposition
1.2]{KLW}). Also the following is immediate (cf. \cite[Proposition 2.1]{KLW}).
\begin{proposition} If $\cR\subset E_1\ti E_2$ is a skew algebroid relation, then
\be\label{mcsar} \mod(\cR)=\ch\left((j_{\cR,1})^!\left(L^{E_1}\right)\ot (j_{\cR,2})^!\left(L^{E_2}\right)^\ast\right)\,.
\ee
Moreover, if $\cR=\graph(\zF)$ is the graph of a skew algebroid morphism $\zF:E_1\ra E_2$, then
$$\mod(\cR)=\mod(E_1)-\zF^\ast(\mod(E_2))\,.$$
\end{proposition}
For a skew algebroid morphism $\zF:E_1\ra E_2$ we will call $\mod(\zF)=\mod(E_1)-\zF^\ast(\mod(E_2))$ the {\it
modular class of the algebroid morphism} $\zF$. Note that, in the case when $\zF$ is a Lie algebroid morphism,
this coincides with the standard definition (see \cite{KSW}).

\begin{example}
Let now $(M_i,\zL_i)$, $i=1,2$, be Poisson manifolds and $\zf:M_1\ra M_2$ be a Poisson map. This means that
the graph of $\zf$ is a coisotropic submanifold in $M_1\ti\ol{M_2}$, where $\ol{M_2}$ is the manifold $M_2$
with the opposite Poisson structure $-\zL_2$ \cite{We0}. It is well known that any Poisson manifold $(M,\zL)$
gives rise to a Lie algebroid structure on $T^\ast M$ associated with the linear Poisson tensor $\dd_T\zL$  on
$TM$, the complete tangent lift of $\zL$ (cf. \cite{GU1,GU3}). The direct product of the tangent bundles
$TM_1\ti\ol{TM_2}$ carries therefore a linear Poisson structure $\zP_1\ti(-\zP_2)$, where $\zP_i=\dd_T\zL_i$,
$i=1,2$, which corresponds to a Lie algebroid structure on $T^\ast M_1\ti\ol{T^\ast M_2}$. Since
$T\graph(\zf)$ is the graph of $T\zf$, it is a coisotropic subbundle in $TM_1\ti\ol{TM_2}$, thus corresponds to a
Lie algebroid relation $\cR\subset T^\ast M_1\ti\ol{T^\ast M_2}$, where
$\cR=\left(T\graph(\zf)\right)^\perp$. It is easy to see that the modular class
$$\mod(\cR)=j_{\cR,1}^\ast(\mod(T^\ast M_1,\Pi_1))-j_{\cR,2}^\ast(\mod(T^\ast M_2,-\Pi_2))$$
coincides with $2\mod(\zf)$, where $\mod{\zf}$ is the modular class of the Poisson map $\zf$ as defined in
\cite{CF}. The factor 2 comes from the fact that the Lie algebroid modular class $\mod(T^\ast M)$ is
$2\mod(M)$, where $\mod(M)$ is the modular class of a Poisson manifold $(M,\zL)$ as defined in \cite{We}.
Actually, what is used in \cite{CF} implicitly is the relation $\cR'=(T\zf)^\ast\subset TM_1\ti{TM_2}$, the
dual relation of $T\zf$, embedded in $TM_1\ti{TM_2}$ by the map $(j_{\cR,1},-j_{\cR,2})$, and
the minus sign is canceled by the sign at $\zP_2$, so the result is the same.
\end{example}
\begin{remark} Of course, in the above example, instead of $\graph(\zf)$ we can start with an arbitrary
coisotropic submanifold in $M_1\ti\ol{M_2}$ which is a Poisson relation between $(M_1,\zL_1)$ and $(M_2,\zL_2)$.
\end{remark}

\section{Acknowledgements}
The author is indebted to Alan Weinstein, Rui Loja Fernandes, and Juan Carlos Marrero for their helpful suggestions
and comments.

\bibliographystyle{plain}

\bigskip
\noindent Janusz Grabowski\\Institute of Mathematics, Polish Academy of Sciences\\\'Sniadeckich 8, P.O. Box
21, 00-956 Warszawa,
Poland\\{\tt jagrab@impan.pl}\\\\

\end{document}